\documentclass{article}

\usepackage{amsmath,amsfonts,amssymb,amsthm}
\usepackage{mathrsfs}
\usepackage{mathtools}

\newtheorem{theorem}{Theorem}

\newtheorem{lemma}[theorem]{Lemma}
\newtheorem{corollary}[theorem]{Corollary}

\theoremstyle{definition}

\newtheorem*{remark}{Remark}

\begin{document}

\title{A New Proof of the Weyl-von Neumann-Berg Theorem}
\markright{A New Proof of the Weyl-von Neumann-Berg Theorem}
\author{
        \small
        \begin{tabular}{cc}
        Longxiang Fan & Shichang Song \\
        School of Mathematical Sciences & School of Mathematics and Statistics \\
        Beijing Normal University & Beijing Jiaotong University \\
        \texttt{longxiangfan@hotmail.com} & \texttt{ssong@bjtu.edu.cn}
        \end{tabular}
        }

\maketitle

\begin{abstract}
We give a new proof of the Weyl-von Neumann-Berg theorem. Our proof improves Halmos' proof in 1972 by observing the fact that every compact set in the complex plane is the continuous image of a compact set in the real line.
\end{abstract}

\noindent
The Weyl-von Neumann-Berg theorem was proved by Berg \cite{Berg1971AnEO} in 1971. In 1972, Halmos \cite{ContinuousFunctionsofHermitianOperators} gave a new proof. Halmos used the Alexandroff-Hausdorff theorem, i.\,e. every compact metric space is the continuous image of the Cantor set. However, the Alexandroff-Hausdorff theorem is rather strong and not necessary. We follow Halmos' proof strategy to give an improved proof by the observation: every compact space in the complex plane is the continuous image of a compact set in the real line. To make our paper self-contained, we will give a detailed proof.

\section{Lemmas from general topology}
In this section we will give several rather elementary results necessary for our new proof of the Weyl-von Neumann-Berg theorem.
\begin{lemma}\label{continuous image of K}
	Every nonempty compact set in the complex plane is the continuous image of a compact set in the real line.
\end{lemma}
\begin{proof}
	Since the complex plane and $\mathbb{R}^2$ are homeomorphic, we consider compact sets in $\mathbb{R}^2$.
	Without loss of generality, let $\Lambda$ be a non-empty compact subset of the unit square in $\mathbb{R}^2$. 
Let $f \colon [0,1] \longrightarrow [0,1]^2$ be the well-known Peano curve which is a continuous surjection.
Let $K$ be the preimage of $\Lambda$ under $f$ and define $\phi:=f|_K$. It is easy to verify that $K$ is compact and $\phi$ is continuous by the compactness and continuity of $f$.
\end{proof}
\begin{remark}
    The construction of a Peano curve is surely a nontrivial result. It is defined recursively. For more details, please check $\cite{Peanocurve}$. As far as we know, there may not be an explicit or simple expression of it.
\end{remark}

Let $X$ be a metric space. A function $f(x)\colon X \to\mathbb{R}$ is lower semicontinuous at the point $x_0\in X$ if $f(x_0)\leq \liminf\limits_{x\to x_0}f(x)$.

\begin{lemma}\label{semicontinuous is Borel}
	Let $f\colon \Lambda \rightarrow \mathbb{R}$ be lower semicontinuous, where $\Lambda$ is a closed subset of $\mathbb{C}$.
	Then $f$ is a Borel function.
\end{lemma}
\begin{proof}
	For all $a \leq b$, we have
	$$\{ a \leq f \leq b \}=\{ a \leq f \} \cap \{ f \leq b \} = (\bigcup\limits_{n=1}^{\infty} \{ f \leq a - \frac{1}{n} \})^{c} \cap \{ f \leq b \}.$$
	By the Baire theorem, $\{ f \leq r \}$ is closed for each $r \in \mathbb{R}$, and thus $\{ a \leq f \leq b \}$ is Borel.
	Hence, $f$ is Borel.
\end{proof}

\begin{lemma}\label{Borel function id}
	Suppose that $\Lambda$ is a non-empty compact set in the complex plane, $K$ is a compact subset of $\mathbb{R}$ and $\phi$ is a continuous function such that $\phi(K) = \Lambda$, then there exists a Borel function $\psi \colon \Lambda \rightarrow K$ such that $\bigl(\phi \circ \psi\bigr) (z) = z $ for all $z \in \Lambda$.
\end{lemma}
\begin{proof}
	Define
	\begin{align*}
		\psi \colon \Lambda &\rightarrow K\\
		z &\mapsto \inf\{ x\in K \mid \phi(x)=z \}.
	\end{align*}
	It is easy to verify that $\{ x\in K \mid \phi(x)=z \}$ is a compact set for all $z \in \Lambda$ by the continuity of $\phi$ and the compactness of $K$. Fix a point $z$ in $\Lambda$ and choose an arbitrary sequence $\{ z_n \}_{n=1}^{\infty}$ that converges to $z$, then there exists a subsequence $\{ \psi(z_{\sigma(n)}) \}_{n=1}^{\infty}$ of $\{ \psi(z_n) \}_{n=1}^{\infty}$ that converges in $K$. Let $x_{\sigma(n)} = \psi(z_{\sigma(n)})$ for all $n \in \mathbb{N}^+$ and let $x$ be the limit point of $\{ x_{\sigma(n)} \}_{n=1}^{\infty}$. Then we have
	$$\bigl(\phi \circ \psi\bigr)(z)=\phi(\inf\{ x\in K \mid \phi(x)=z \}) = z$$
	Meanwhile, the continuity of $\phi$ implies that
	$$z=\lim\limits_{n \to \infty}z_{\sigma(n)}=\lim\limits_{n\to\infty}\phi(x_{\sigma(n)})=\phi(x)$$
	i.\,e., $\psi(z)=\inf\{ x\in K \mid \phi(x) =z  \} \leq x$. Thus $\psi(z) \leq \liminf_{n \rightarrow \infty}\psi(z_n)$, which follows that $\psi$ is lower semicontinuous, and thus Borel by Lemma \ref{semicontinuous is Borel}.
\end{proof}

\section{New Proof of Weyl-von Neumann-Theorem}
Now we start our new proof of the Weyl-von Neumann-Berg theorem.

\begin{theorem}\label{Every bounded normal operator on a Hilbert space is the continuous image of a Hermitian one}
	Every bounded normal operator on a Hilbert space is the continuous image of a Hermitian one.
\end{theorem}
\begin{proof}
	Let $A$ be a bounded normal operator on a Hilbert space $\mathscr{H}$ and $\Lambda$ be the spectrum of $A$, then there exists a non-empty conpact set $K \subseteq \mathbb{R}$, a continuous function $\phi \colon K \rightarrow \Lambda$ and a Borel function $\psi$ such that $\phi(K)=\Lambda$ and $\phi \circ \psi$ is the identity map on $\Lambda$ by Lemma \ref{continuous image of K} and Lemma \ref{Borel function id}.
	\par
	Without loss of generality, suppose that $A$ has a cyclic vector in $\mathscr{H}$, i.\,e., there exists a vector $x \in \mathscr{H}$ such that $\mathrm{span}\{ A^n x \}_{n=1}^{\infty}$ is dense in $\mathscr{H}$, then there exists a measure $\mu \in [0,+\infty)$ on $\Lambda$ and a unitary operator $U$ from $L^2(\mu)$ onto $\mathscr{H}$ such that for all $f \in L^2(\mu)$, $(U^{-1}AU)(f)=\mathrm{id}_{\Lambda} \cdot f$ by the spectral theorem; see \cite{WhatDoestheSpectralTheoremSay}. Let $\gamma=\mu \circ \psi^{-1}$ be the induced measure on $K$ \cite[Chapter 6, Theorem 1.5]{RealAnalysis}. Define
	\begin{align*}
		T \colon L^2(\gamma) &\rightarrow L^2(\mu)\\
		g &\mapsto g \circ \psi.
	\end{align*}
	We are going to prove that $T$ is a unitary operator by change of variables in Lebesgue integrals \cite[\S 39]{MeasureTheory}. Indeed we have
	$$\|Tg\|^2 = \int_{\Lambda} |g \circ \psi|^2 d\mu = \int_{K} |g|^2 d(\mu \circ \psi^{-1}) = \int_{K} |g|^2 d\gamma = \|g\|^2$$
	and
	$$(Tg,f) = \int_{\Lambda} (g \circ \psi) \cdot \overline{(f \circ \phi \circ \psi)} d\mu = \int_{K} g \cdot \overline{(f \circ \phi)} d(\mu \circ \psi^{-1}) = \int_{K} g \cdot \overline{(f \circ \phi)} d\gamma,$$
	so the adjoint operator $T^*$ of $T$ is that $T^*f=f \circ \phi$.
	\par
	Moreover, we have that $T(T^*f) = T(f \circ \phi) = f \circ \phi \circ \psi = f$, which implies that every $f \in L^2(\mu)$ has a preimage $T^*f$, i.\,e., $T$ is a surjective isometry, and thus unitary.
	\par
	To make our statement simple, suppose that $A$ is the normal operator on $L^2(\mu)$ by the isometry between $\mathscr{H}$ and $L^2(\mu)$, then we have
	\begin{align*}
		(T(\phi \cdotp T^*f))(z)&=(T(\phi \cdotp (f \circ \phi)))(z)=((\phi \cdotp (f \circ \phi))\circ \psi)(z) \\
		&=\phi(\psi(z)) \cdotp f(\phi(\psi(z)))=zf(z)
	\end{align*}
	We can easily get that $A(g)=\phi \cdotp g$ if we write $g=T^*f$ and suppose that $A$ is the normal operator on $L^2(\gamma)$ by using the isometry $T$. Define
	\begin{align*}
		B \colon L^2(\gamma) &\rightarrow L^2(\gamma)\\
		g &\mapsto \mathrm{id}_{K} \cdot g
	\end{align*}
	then $B$ is a Hermitian operator on $L^2(\gamma)$ and $\phi(B)=A$. (All we need to do is to verify that $\phi(B)=A$ if $\phi$ is a polynomial and use the Weierstrass approximation theorem and the Gelfand-Naimark theorem \cite[(1.20) The Gelfand-Naimark Theorem]{ACourseInAbstractHarmonicAnalysis}.) This completes our proof.
\end{proof}

\begin{corollary}[Weyl-von Neumann-Berg Theorem]\label{Weyl-von Neumann-Berg theorem}
	Every bounded normal operator on a separable Hilbert space is the sum of a diagonal one and a compact one.
\end{corollary}
\begin{proof}
	Suppose that $A$ is a bounded normal operator on a separable Hilbert space. By Theorem \ref{Every bounded normal operator on a Hilbert space is the continuous image of a Hermitian one}, there are a Hermitian operator $B$ and a continuous function $\phi$ such that $\phi(B)=A$. By the Weyl-von Neumann theorem \cite{Neumann1935CharakterisierungDS}, we can get $B=D+C$ where $D$ is a diagonal operator and $C$ is a compact one. Extend $\phi$ to a continuous function which is defined on a compact set that includes the spectra of $B$ and $D$. The Gelfand-Naimark theorem states that such extension exists, and $\phi(B)$ and $\phi(D)$ are bounded linear operators.
	\par
	By the Weierstrass approximation theorem, there exists a sequence of polynomials $\{ p_n \}_{n=1}^{\infty}$ that converges to $\phi$ uniformly, which implies that
	$$\lim\limits_{n \rightarrow \infty} p_n(B) = \phi(B) \, , \, \lim\limits_{n \rightarrow \infty} p_n(D) = \phi(D).$$
	\par
	Obviously $\phi(D)$ is diagonal. Let $C_n=p_n(B)-p_n(D)=p_n(D+C)-p_n(D)$.
	Since $(D+C)^k - D^k = D^k + D^{k-1}C + \cdots + C^k - D^k =  D^{k-1}C + \cdots + C^k$ for each $k\in\mathbb{N}^+$, we have that $(D+C)^k - D^k$ is compact, and thus $C_n$ is compact because the set of all compact operators is an ideal in the bounded linear operator space.
	Let $L:=\phi(B)-\phi(D)$. Then $L=\lim_{n\to\infty}C_n$. By the closedness of the ideal of compact operators, $L$ is compact. Then $A=\phi(B)=\phi(D)+L$, which completes our proof.
\end{proof}

\begin{remark}
	Our proof of the above corollary is completely the same as Halmos' proof in \cite[Corollary]{ContinuousFunctionsofHermitianOperators}. We just want to make our paper self-contained.
\end{remark}

\textbf{Acknowledgment.} This work is part of Longxiang Fan's bachelor thesis at the Beijing Jiaotong University supervised by Shichang Song in 2022. We are grateful to Jean-Christophe Bourin for pointing our mistake on the construction of the Peano curve.

\bibliographystyle{amsalpha}
\bibliography{references}

\end{document}